\pdfoutput=1
\documentclass[a4paper,12pt]{amsart} 

\usepackage[DIV=15]{typearea} 

\usepackage[UKenglish]{babel}
\usepackage[T1]{fontenc}

\usepackage{amsmath}
\usepackage{amsthm}
\usepackage{amsfonts}
\usepackage{amssymb}
\usepackage[neverdecrease]{paralist}
\usepackage{enumitem}
\usepackage{dsfont}
\usepackage{ifthen}
\usepackage{twoopt}
\usepackage{mathtools}
\usepackage{etoolbox}
\usepackage{xspace}

\usepackage{microtype}
\newcommand{\mathdcl}[1]{{\ifstrequal{#1}{l}{l}{\textup{#1}}}}

\newcommand{\RR}{\mathds{R}}

\newcommand{\CC}{\mathds{C}}
\newcommand{\one}{\mathds{1}}
\newcommand{\setmid}{:}
\newcommand{\m}{\mathdcl{m}}
\newcommand{\reg}{\mathdcl{reg}}
\newcommand{\sing}{\mathdcl{s}}
\newcommand{\cpt}{\mathdcl{c}}
\newcommand{\loc}{\mathdcl{loc}}

\newcommand{\form}[1]{{\mathfrak{#1}}}
\newcommand{\formdom}[1]{D(\form{#1})}
\DeclareMathOperator{\formreal}{\mathfrak{Re}}
\DeclareMathOperator{\formimag}{\mathfrak{Im}}
\renewcommand{\Re}{\operatorname{Re}}
\renewcommand{\Im}{\operatorname{Im}}

\newcommand{\conj}[1]{\overline{#1}}
\DeclarePairedDelimiter\norm{\lVert}{\rVert}
\DeclarePairedDelimiter\abs{\lvert}{\rvert}

\newcommand{\set}[3][auto]{{% KEEP:
\ifthenelse{\equal{#1}{auto}}{\left\{#2\setmid #3\right\}}{}
\ifthenelse{\equal{#1}{b}}{\bigl\{#2\setmid #3\bigr\}}{}
\ifthenelse{\equal{#1}{B}}{\Bigl\{#2\setmid #3\Bigr\}}{}
}}
\newcommand{\scalar}[3][auto]{{% KEEP:
\ifthenelse{\equal{#1}{auto}}{\left(#2\mkern3mu{\mid}\mkern3mu #3\right)}{}
\ifthenelse{\equal{#1}{b}}{\bigl(#2\mkern3mu{\mid}\mkern3mu #3\bigr)}{}
\ifthenelse{\equal{#1}{B}}{\Bigl(#2\mkern3mu{\bigm|}\mkern3mu #3\Bigr)}{}
}}
\newcommand{\ascalar}[3][auto]{{% KEEP:
\ifthenelse{\equal{#1}{auto}}{\left<#2, #3\right>}{}
\ifthenelse{\equal{#1}{b}}{\bigl<#2, #3\bigr>}{}
\ifthenelse{\equal{#1}{B}}{\Bigl<#2, #3\Bigr>}{}
}}

\newcommand{\Linop}{\mathcal{L}}

\newcommand{\Cinfc}[1][\Omega]{\ensuremath{C^\infty_\cpt(#1)}}
\newcommand{\Lone}[1][\Omega]{\ensuremath{L^1(#1)}}
\newcommand{\Ltwo}[1][\Omega]{\ensuremath{L^2(#1)}}
\newcommand{\Linf}[1][\Omega]{\ensuremath{L^\infty(#1)}}
\newcommand{\Loneloc}[1][\Omega]{\ensuremath{L^1_\loc(#1)}}
\newcommand{\Linfloc}[1][\Omega]{\ensuremath{L^\infty_\loc(#1)}}
\newcommand{\Hone}[1][\Omega]{\ensuremath{{H^1(#1)}}}

\newcommand{\restrict}[2]{\ensuremath{#1|_{#2}}}
\newcommand{\meas}[2][Lebesgue]{% KEEP:
\ifthenelse{\equal{#1}{Lebesgue}}{\abs{#2}}{{#1(#2)}}% KEEP:
}
\newcommand{\emphdef}[1]{\textbf{\boldmath #1\unboldmath}}

\newenvironment{romanenum}{%
\begin{enumerate}[label=\textup{(\roman*)},ref=\textup{(\roman*)},itemsep=0em,topsep=1ex]%
}{\end{enumerate}}

\makeatletter
\newcommand\sauter@restoreitemindent{%
\itemindent=\sauter@saveditemindent
\def\sauter@restoreitemindent{}}
\newdimen\sauter@savedparindent

\newenvironment{parenum}[1][]%
{\sauter@savedparindent=\parindent%
\ifthenelse{\equal{#1}{}}{\asparaenum}{\asparaenum[#1]}%
\edef\sauter@saveditemindent{\the\itemindent}%
\advance\itemindent-\sauter@savedparindent
\patchcmd{\@item}{\ignorespaces}{\sauter@restoreitemindent\ignorespaces}{}{}}
{\endasparaenum}
\makeatother

\theoremstyle{plain}
\newtheorem{theorem}{Theorem}[section]
\newtheorem{proposition}[theorem]{Proposition}

\newtheorem{lemma}[theorem]{Lemma}
\theoremstyle{definition}
\newtheorem{example}[theorem]{Example}
\newtheorem{remark}[theorem]{Remark}
\newtheorem{definition}[theorem]{Definition}

\renewcommand{\jmath}{j}

\title[The regular part of second-order differential sectorial forms]{The regular part of second-order differential sectorial forms\\ with lower-order terms}

\author{\sc A.F.M. ter~Elst}
\address{A.F.M. ter~Elst\\Department of Mathematics\\The University of Auckland\\Private bag 92019\\Auckland 1142\\New Zealand}
\email{terelst@math.auckland.ac.nz}

\author[M. Sauter]{\sc Manfred Sauter}
\address{Manfred Sauter\\Department of Mathematics\\The University of Auckland\\Private bag 92019\\Auckland 1142\\New Zealand}
\email{m.sauter@auckland.ac.nz}

\date{March 2013}


\begin{document}

\begin{abstract}
We present a formula for the regular part of a sectorial form that represents a general linear 
second-order differential expression that may include lower-order terms. 
The formula is given in terms of the original coefficients. It shows that
the regular part is again a differential sectorial form and allows
to characterise when also the singular part is sectorial.
While this generalises earlier results on pure second-order differential expressions,
it also shows that lower-order terms truly introduce new behaviour.
\end{abstract}

\begingroup
% Hack to get rid of the amsart all-caps style in the title.
% http://tex.stackexchange.com/questions/2820/disable-toggle-smallcaps
\makeatletter
\patchcmd{\@settitle}{\uppercasenonmath\@title}{\large}{}{}
\patchcmd{\@setauthors}{\MakeUppercase}{}{}{}
\makeatother
\maketitle
\endgroup

\section{Introduction}

We study the regular part of a differential sectorial form that represents a
general linear second-order differential expression that may include lower-order terms. 
Loosely speaking, such a differential expression has the form
\[
    -\sum_{k,l=1}^d \partial_lc_{kl}\partial_k + \sum_{k=1}^d b_k\partial_k - \sum_{k=1}^d \partial_k d_k + c_0.
\]
We obtain a formula for the regular part of the corresponding differential sectorial form
from which it follows that
the regular part is again a differential sectorial form. Furthermore, this formula
allows to characterise when the singular part is sectorial and when the 
regular part of the real part is equal to the real part of the regular part.

The regular part for positive symmetric forms was introduced in~\cite{bSim5} and recently generalised to sectorial forms in~\cite{AtE12:sect-form}.
In~\cite{Vog1} a formula was obtained for the regular part of a positive symmetric form representing a pure second-order differential expression.
We make use of the techniques introduced there, as we already did in a previous paper~\cite{tES2011}.

A different approach to the regular part of a positive symmetric form based on parallel sums is
given in~\cite{HSdS09}. Furthermore, we point out that in the context of
nonlinear phenomena and discontinuous media the relaxation of a functional
is a notion that in the linear setting corresponds to the regular part of positive symmetric forms~\cite{Mosco94,Braides02,DalMaso93}.

The current article generalises our exposition for differential sectorial forms of pure second order as given in~\cite{tES2011}. 
We note that here we impose the same mild conditions on the highest-order coefficients. 
In particular, we cover a large class of forms associated with important degenerate elliptic differential operators.
The lower-order terms introduce new phenomena. Even if $c_0=0$, then the regular part of the corresponding
differential sectorial form may have a non-vanishing zeroth-order term.
Moreover, this term can cause a shift in the vertex of the singular part of the form.
We will present an example where this happens.

\section{Preliminaries and definition of considered form}\label{sec:def-form}

We will adopt the notation as in~\cite{tES2011}. For the reader's convenience, we recall the following definitions.
Suppose $H$ is a Hilbert space, $D(\form{b})$ is a linear subspace of $H$ and
$\form{b}\colon D(\form{b})\times D(\form{b})\to\CC$ is a sesquilinear form with form domain $D(\form{b})$.
We say that the form $\form{b}$ is \emphdef{sectorial} if
there exist a \emphdef{vertex} $\gamma\in\RR$ and a \emphdef{semi-angle} $\theta\in[0,\tfrac{\pi}{2})$ such that
\[
    \form{b}(u,u)-\gamma \, \norm{u}^2_H\in \Sigma_\theta
\]
for all $u\in D(\form{b})$, where
$\Sigma_\theta\coloneqq \set{z\in\CC}{\text{$z=0$ or $\abs{\arg(z)}\le\theta$}}$.
The \emphdef{real part} $\formreal\form{b}\colon D(\form{b})\times D(\form{b})\to\CC$
of a sesquilinear form $\form{b}$ is defined by $(\formreal\form{b})(u,v)=\frac{1}{2}\bigl(\form{b}(u,v) + \conj{\form{b}(v,u)}\bigr)$, 
and similarly for the \emphdef{imaginary part}.

For a sectorial form $\form{b}$ with vertex $\gamma$ we define the norm $\norm{\cdot}_\form{b}$ on $D(\form{b})$ by setting
$\norm{u}^2_\form{b}\coloneqq (\formreal \form{b})(u,u) + (1-\gamma)\norm{u}^2_H$.
A sectorial form $\form{b}$ is called \emphdef{closed} if $(D(\form{b}),\norm{\cdot}_\form{b})$ is a Hilbert space.
A sectorial form is called \emphdef{closable} if it has a closed extension.
Our main reason to study sectorial forms is that with every densely defined sectorial form $\form{b}$ one can associate an \m-sectorial operator $B$ in a natural way by~\cite[Theorem~1.1]{AtE12:sect-form}.

By~\cite[Proposition~2.3]{tES2011} the following definition of the regular part makes sense.
\begin{definition}
Let $\form{b}$ be a densely defined sectorial form. Let $B$ be the $\m$-sectorial operator associated with $\form{b}$.
Then there exists a unique closable sectorial form $\form{b}_\reg$ with form domain $D(\form{b}_\reg)=D(\form{b})$ such that $B$ is associated with $\form{b}_\reg$.
We call $\form{b}_\reg$ the \emphdef{regular part} of $\form{b}$.
The \emphdef{singular part} of $\form{b}$ is the form $\form{b}_\sing=\form{b}-\form{b}_\reg$.
\end{definition}
A useful formula for the regular part of a densely defined sectorial form is given in~\cite[Theorem~2.6]{tES2011}. 

We now introduce the form $\form{a}$ that is considered throughout this paper.
Let $\Omega\subset\RR^d$ be open. Let $H=L^2(\Omega)$ and suppose $D(\form{a})$ is a vector subspace of $H$ that contains $\Cinfc$. Suppose that $\partial_ku\in\Loneloc$ for all $k\in\{1,\dotsc,d\}$
and $u\in D(\form{a})$.
For all $k,l\in\{1,\dotsc,d\}$, let $c_{kl}$, $b_k$, $d_k$ and $c_0$ be measurable functions from $\Omega$ into $\CC$.
Suppose that $c_0\in\Linf$.
Define $C\colon\Omega\to\CC^{d\times d}$ by $C(x)=(c_{kl}(x))_{k,l=1}^d$ and $b,d\colon\Omega\to\CC^d$ by $b(x)=(b_k(x))_{k=1}^d$ and $d(x)=(d_k(x))_{k=1}^d$.
Suppose that $C\nabla u\cdot\nabla u\in\Lone$ for all $u\in D(\form{a})$, where $\zeta\cdot\eta$ denotes the Euclidean inner product of $\zeta,\eta\in\CC^d$.
Moreover, suppose there exists a $\theta\in[0,\pi/2)$ such that $C(x)\xi\cdot\xi\in\Sigma_\theta$ for all $\xi\in\CC^d$ and $x\in\Omega$.
Define the measurable function $A\colon\Omega\to\CC^{d\times d}$ by $A(x)=\frac{1}{2}(C(x)+C(x)^*)$. Note that $A(x)$ is a positive semi-definite Hermitian matrix for all $x\in\Omega$.
Hence $A(x)$ admits a unique positive semi-definite square root $A^{1/2}(x)$ for all $x\in\Omega$.
Furthermore, suppose there exists a $K>0$ such that
\[
    \abs{\conj{b(x)}\cdot\xi} \le K\norm{A^{1/2}(x)\xi}_{\CC^d}\qquad\text{and}\qquad\abs{d(x)\cdot\xi}\le K\norm{A^{1/2}(x)\xi}_{\CC^d}
\]
for all $\xi\in\CC^d$ and $x\in\Omega$.
The map $x\mapsto A^{1/2}(x)$ is measurable from $\Omega$ into $\CC^{d\times d}$ by~\cite[Lemma~4.1]{tES2011}.
\begin{lemma}\label{lem:XY}
There exist measurable, bounded maps $X,Y\colon \Omega\to\CC^{d}$ such that
$A^{1/2}(x) X(x)=\conj{b(x)}$ and $A^{1/2}(x)Y(x)=d(x)$ for all $x\in\Omega$.
\end{lemma}
\begin{proof}
Define $g\colon [0,\infty)\to[0,\infty)$ by $g(t)=t^{-1/2}$ if $t>0$ and $g(0)=0$.
Then, as in the proof of~\cite[Lemma~4.1]{tES2011}, the map $x\mapsto g(A(x))$ is measurable from $\Omega$ into $\CC^{d\times d}$.
Therefore also the map $X\colon\Omega\to\CC^{d}$ defined by $X(x)=g(A(x))\conj{b(x)}$ is measurable.
Moreover,
\[
    \abs{X(x)\cdot\xi}=\abs{\conj{b(x)}\cdot g(A(x))\xi} \le K\norm{A^{1/2}(x)g(A(x))\xi}_{\CC^d} \le K\norm{\xi}_{\CC^d}.
\]
This proves that $X$ is bounded. 
Then arguing as in the proof of the first displayed formula on page~918 in~\cite{tES2011}, one deduces $A^{1/2}(x)X(x)=\conj{b(x)}$.

Existence and boundedness of $Y$ is proved similarly.
\end{proof}

Let $Z\colon\Omega\to\CC^{d\times d}$ be as in~\cite[Lemma~4.1]{tES2011}.
Then $Z$ is bounded, $Z(x)$ is Hermitian and $A^{1/2}(x)(I+iZ(x))A^{1/2}(x)=C(x)$ for all $x\in\Omega$.
Finally, define the form $\form{a}\colon D(\form{a})\times D(\form{a})\to\CC$ by
\begin{equation}\label{eq:def-a}
    \form{a}(u,v)=\int_\Omega\sum_{k,l=1}^d c_{kl}(\partial_k u)\conj{\partial_l v} 
      + \int_\Omega\sum_{k=1}^d b_k(\partial_k u)\conj{v} 
      + \int_\Omega\sum_{k=1}^d d_k u\,\conj{\partial_k v}
      + \int_\Omega c_0 u\conj{v}.
\end{equation}
Let $X$ and $Y$ be as in Lemma~\ref{lem:XY}.
This allows us to write
\begin{equation}\label{eq:simpform}
    \form{a}(u,v)=\scalar{(I+iZ)A^{1/2}\nabla u}{A^{1/2}\nabla v}
      + \scalar{A^{1/2}\nabla u}{vX}
      + \scalar{uY}{A^{1/2}\nabla v}
      + \scalar{c_0u}{v}
\end{equation}
for all $u,v\in D(\form{a})$. In particular, it follows that the first-order terms in~\eqref{eq:def-a} are indeed integrable.

Since $X$, $Y$ and $Z$ are bounded, the next lemma follows from~\eqref{eq:simpform}.
\begin{lemma}
The form $\form{a}$ is a sectorial form in $H$.
\end{lemma}

Let $\form{b}$ be a sesquilinear form in $\Ltwo$.
If $\form{b}$ is equal to the form $\form{a}$ for
an appropriate choice of the coefficient functions $c_{k,l}$, $b_k$, $d_k$ and $c_0$ with $k,l\in\{1,\dotsc,d\}$,
then we shall call $\form{b}$ a \emphdef{differential sectorial form}.
The main result of this paper establishes that, under suitable mild conditions
which we next introduce, the regular part of $\form{a}$ is also a differential sectorial form.

As in~\cite{tES2011}, we introduce the following two conditions.
We say that
$\form{a}$ satisfies 
\emphdef{Condition~\upshape{(L)}} if 
\begin{romanenum}
\item $\formdom{a}\cap\Linf$ is invariant under multiplication with $\Cinfc$ functions, 
\item there exists a $\psi\in C^1_\mathdcl{b}(\RR)$ such that $\psi(0)=0$, 
$\psi'(0)=1$, and $\psi\circ(\Re u)+i \psi\circ(\Im u)\in\formdom{a}$
for all $u\in\formdom{a}$, and,
\item
$c_{kl} + \conj{c_{lk}}$ is real valued for all $k,l \in \{ 1,\dotsc,d \} $.
\end{romanenum}
This is the condition on the form $\form{a}$ in \cite{Vog1}, adapted to complex
vector spaces.
Furthermore, we say that $\form{a}$ satisfies \emphdef{Condition~\upshape{(B)}} if
\begin{romanenum}
\item $\formdom{a}$ is invariant under multiplication with $\Cinfc$ functions, and,
\item $c_{kl} \in \Linfloc$ for all $k,l \in \{ 1,\dotsc,d \} $.
\end{romanenum}

\section{The formula for the regular part}

In this section we derive a formula for the regular part of the form $\form{a}$ in Section~\ref{sec:def-form}.
To this end, we assume that $\form{a}$ satisfies Condition~\upshape{(L)} or~\upshape{(B)}.
It will be immediate from the obtained formula that both the regular
and singular part of $\form{a}$ continue to be \emph{differential} sectorial forms. 
Note that this is not at all clear as the definition of the regular part is rather abstract.

We will refine the methods used in~\cite[Section~4]{tES2011}.
We first reformulate some of the results in~\cite[Section~2]{tES2011} in a convenient way for the current setting.

If $V_0$ is a vector space with a semi-inner product, then there exist a Hilbert space $V$ and an isometric linear map $\Phi\colon V_0\to V$
such that $\Phi(V_0)$ is dense in $V$. Then $(V,\Phi)$ is unique up to unitary equivalence. We call $(V,\Phi)$ the \emphdef{completion} of $V_0$.
Moreover, every uniformly continuous map on $V_0$ has a unique continuous extension to $V$ in the obvious way.
\begin{proposition}\label{prop:areg-abstract}
Let $H$ be a Hilbert space and $\form{a}$ be a densely defined sectorial form in $H$.
Let $(V,\Phi)$ be the completion of $(\formdom{a},\norm{\cdot}_{\form{a}})$.
Let $\form{\tilde{a}}\colon V\times V\to\CC$ and $\tilde{j}\colon V\to H$ be the continuous extensions of 
$\form{a}$ and of the embedding of $\formdom{a}$ into $H$, respectively. Let $\form{\tilde{h}}$ be
the real part of $\form{\tilde{a}}$.
Let $\pi_1$ be the orthogonal projection of $V$ onto $V_1\coloneqq\ker\tilde{j}$ and let $\pi_2=I_{V}-\pi_1$.
Then there exists a unique operator $T\in\Linop(V,V_1)$ such that
\[
    (\formimag\form{\tilde{a}})(u,v)=\form{\tilde{h}}(Tu,v)
\]
for all $u\in V$ and $v\in V_1$. Moreover, $T_{11}\coloneqq\restrict{T}{V_1}\in\Linop(V_1)$ is
self-adjoint. Define the operator $\Pi\in\Linop(V)$ by
\begin{equation}\label{eq:PVa}
    \Pi u = \pi_2 u - i(I_{V_1} + i T_{11})^{-1}T\pi_2 u.
\end{equation}
Then the regular part of $\form{a}$ is given by
\begin{equation}\label{eq:areg-gen}
    \form{a}_\reg(u,v) = \form{\tilde{a}}(\Pi\Phi(u),\Pi\Phi(v))
\end{equation}
for all $u,v\in\formdom{a}$.
\end{proposition}
\begin{proof}
In~\cite{tES2011} the operator $\Pi$ was denoted by $P_{V(\form{\tilde{a}})}$.
Now~\eqref{eq:PVa} and~\eqref{eq:areg-gen} follow from the displayed formula on the top of page 912 and~(2) in~\cite{tES2011}.
It is easily verified that $\pi_2$ is the orthogonal projection in $V$ onto
\[
    \{u\in V: \text{$\form{\tilde{h}}(u,v)=0$ for all $v\in\ker\tilde{j}$}\}.
\]
Then the remaining statements can be found in~\cite[Theorem~2.6]{tES2011}.
\end{proof}

We start by constructing a suitable completion of the pre-Hilbert space $(\formdom{a},\norm{\cdot}_{\form{a}})$
that allows us to 
get hold of the corresponding continuous extensions of 
$\form{a}$ and of the embedding of $\formdom{a}$ into $H$. 

Let $\form{h}$ be the real part of $\form{a}$. Then
\begin{multline*}
    \form{h}(u,v)=\scalar{A^{1/2}\nabla u}{A^{1/2}\nabla v} 
        + \tfrac{1}{2}\scalar{A^{1/2}\nabla u}{v(X+Y)} \\
        {} + \tfrac{1}{2}\scalar{u(X+Y)}{A^{1/2}\nabla v}
        + \scalar{(\Re c_0)u}{v}
\end{multline*}
for all $u,v\in\formdom{a}$. 
Let $\mathcal{H}$ be the Hilbert space $\Ltwo\times(\Ltwo)^d$ with the usual inner product.
Let $\gamma_0\in\RR$ be a vertex of the sectorial form $\form{a}$.
Since $X$ and $Y$ are bounded, there exists a $\gamma\le\gamma_0$ such that 
the sesquilinear form $\ascalar{\cdot}{\cdot}\colon\mathcal{H}\times\mathcal{H}\to\CC$ defined by
\begin{multline*}
    \ascalar{(u_1,w_1)}{(u_2,w_2)} = \scalar{w_1}{w_2}
        + \tfrac{1}{2}\scalar{w_1}{u_2(X+Y)} \\
        {} + \tfrac{1}{2}\scalar{u_1(X+Y)}{w_2}
        + \scalar{(1-\gamma + \Re c_0)u_1}{u_2}
\end{multline*}
defines an equivalent inner product on $\mathcal{H}$. 
Note that $\gamma$ is also a vertex of $\form{a}$.

Let $\mathcal{H}'$ denote the space $\Ltwo\times(\Ltwo)^d$ 
equipped with the inner product $\ascalar{\cdot}{\cdot}$.
Define the map $\Phi\colon(\formdom{a},\norm{\cdot}_\form{a})\to\mathcal{H}'$ by
\[
    \Phi(u) = (u, A^{1/2}\nabla u).
\]
Then $\Phi$ is an isometry.
Hence the completion $V$ of $\formdom{a}$ can be realised as the closure of $\Phi(\formdom{a})$ in $\mathcal{H}'$ 
equipped with the inner product of $\mathcal{H}'$.
Note that $V$ is also closed in $\mathcal{H}$.
The embedding of $\formdom{a}$ into $H$ has as continuous extension to the completion $V$
the map $\tilde{j}\colon V\to H$ given by $\tilde{j}(u,w)=u$.
Furthermore, due to~\eqref{eq:simpform} the continuous extension of the form $\form{a}$
is the form $\form{\tilde{a}}\colon V\times V\to\CC$ given by
\begin{equation}\label{eq:atilde}
    \form{\tilde{a}}((u_1,w_1),(u_2,w_2)) = \scalar{(I+iZ)w_1}{w_2} + \scalar{w_1}{u_2X} + \scalar{u_1Y}{w_2} + \scalar{c_0u_1}{u_2}.
\end{equation}
Then the real part $\form{\tilde{h}}$ of $\form{\tilde{a}}$ is given by
\begin{equation}
\label{eq:form-ht}
    \form{\tilde{h}}((u_1,w_1),(u_2,w_2)) = \scalar{w_1}{w_2} + \tfrac{1}{2}\scalar{w_1}{u_2(X+Y)} + \tfrac{1}{2}\scalar{u_1(X+Y)}{w_2} + \scalar{(\Re c_0)u_1}{u_2}.
\end{equation}
Note that
\[
    \ascalar{(u_1,w_1)}{(u_2,w_2)} = \scalar{(u_1,w_1)}{(u_2,w_2)}_{\form{\tilde{h}}}
\]
for all $(u_1,w_1), (u_2,w_2)\in V$.

Next, set $V_1\coloneqq \ker\tilde{j}$ and $V_2\coloneqq V_1^\perp$. Let $\pi_1$ and $\pi_2$ be the orthogonal projections in $V$ onto $V_1$ and $V_2$, respectively.
We shall show that $\pi_1$ and $\pi_2$ can be represented by multiplication operators in $\mathcal{H}$.
To this end, let $V_\sing$ be such that $V_\sing\subset (\Ltwo)^d$ and $V_1=\{0\}\times V_\sing$.
Clearly $V_\sing$ is a closed subspace of $(\Ltwo)^d$. Moreover, since $\form{a}$ satisfies Condition~\upshape{(L)} or~\upshape{(B)},
the orthogonal projection in $(\Ltwo)^d$ onto $V_\sing$ is given by the multiplication operator associated with a measurable function $Q\colon\Omega\to\CC^{d\times d}$
that has values in the orthogonal projection matrices. 
This result is based on~\cite[Proof of Theorem~1]{Vog1}, see also the discussion in~\cite[Section~4]{tES2011}.
Now we are able to state the main result of this paper.
\begin{theorem}\label{thm:formula}
Let $\form{a}$ be defined as in Section~\ref{sec:def-form}. Assume that $\form{a}$ satisfies Condition~\textup{(L)} or~\textup{(B)}.
Let $Q\colon\Omega\to\CC^{d\times d}$ be as before and set $W\coloneqq Q(I+iQZQ)^{-1}Q$ and $P\coloneqq I-Q$.
Then the regular part of $\form{a}$ is given by
\begin{equation}\label{eq:genformula}
\begin{split}
    \form{a}_\reg(u,v)
        &=\scalar{(I+iZ + ZW Z)PA^{1/2}\nabla u}{PA^{1/2}\nabla v} \\
        &\quad {} + \scalar[b]{\conj{X}^{\,\mathdcl{t}}(I-iWZ)PA^{1/2}\nabla u}{v} 
        + \scalar[b]{u}{\conj{Y}^{\,\mathdcl{t}}(I+iW^*Z)PA^{1/2}\nabla v} \\
        &\quad {} - \scalar[b]{(\conj{X}^{\,\mathdcl{t}} WY)u}{v} + \scalar{c_0u}{v}.
\end{split}
\end{equation}
\end{theorem}
\begin{proof}
We will use and adopt the notation of Proposition~\ref{prop:areg-abstract}.
First we represent the operators $\pi_2$, $T\pi_2$ and $(I_{V_1}+iT_{11})^{-1}$ 
by multiplication operators.

Observe that $(u,w)\mapsto (0,w+\tfrac{1}{2}u(X+Y))$ is the orthogonal projection of $\mathcal{H}'$ onto $\{0\}\times (\Ltwo)^d$.
Moreover, considering $\{0\}\times(\Ltwo)^d$ as a (closed) subspace of $\mathcal{H}'$, the map $(0,w)\mapsto (0,Qw)$ is the orthogonal projection of $\{0\}\times (\Ltwo)^d$ onto $V_1$.
Since $V_1\subset V\subset\mathcal{H'}$, the map $\pi_1$ is given by
\[
     \pi_1(u,w) = \bigl(0,Qw + \tfrac{1}{2}uQ(X+Y)\bigr)
\]
for all $(u,w)\in V$. Therefore
\[
     \pi_2(u,w) = \bigl(u, (I-Q)w - \tfrac{1}{2}uQ(X+Y)\bigr)
\]
for all $(u,w)\in V$.

Let $(u_1,w_1), (u_2,w_2)\in V$.
It follows from~\eqref{eq:atilde} that
\begin{multline*}
    (\formimag\form{\tilde{a}})\bigl((u_1,w_1),(u_2,w_2)\bigr) = \scalar{Zw_1}{w_2} 
        - \tfrac{i}{2}\scalar{w_1}{u_2(X-Y)} \\ {} + \tfrac{i}{2}\scalar{u_1(X-Y)}{w_2}
        + \scalar{(\Im c_0)u_1}{u_2}.
\end{multline*}
So if $(u_2,w_2)\in V_1$, then $u_2=0$ and
\begin{align*}
    (\formimag\form{\tilde{a}})\bigl((u_1,w_1),(0,w_2)\bigr) 
        &= \scalar{Zw_1}{w_2} + \tfrac{i}{2}\scalar{u_1(X-Y)}{w_2} \\
        &= \scalar{QZw_1 + \tfrac{i}{2} u_1Q(X-Y)}{w_2} \\
        &= \form{\tilde{h}}\bigl((0,QZw_1+\tfrac{i}{2}u_1 Q(X-Y)), (0,w_2)\bigr).
\end{align*}
Hence the operator $T\in\Linop(V,V_1)$ in Proposition~\ref{prop:areg-abstract} is given by 
\[
    T(u,w)=\bigl(0,QZw+\tfrac{i}{2}uQ(X-Y)\bigr)
\]
for all $(u,w)\in V$.
Then $T_{11}=\restrict{T}{V_1}$ is given by $T_{11}(0,w)=(0,QZw)$ for all $(0,w)\in V_1$.
So as in~\cite[proof of~(13)]{tES2011} we have
\[
    (I_{V_1}+iT_{11})^{-1}(0,w) = \bigl(0,(I+iQZQ)^{-1}w\bigr) = (0,Ww)
\]
for all $(0,w)\in V_1$.
Next note that
\begin{equation}\label{eq:Tpi2}
   T\pi_2(u,w) = \bigl(0, QZ(I-Q)w - \tfrac{1}{2}uQZQ(X+Y) + \tfrac{i}{2}uQ(X-Y)\bigr)
\end{equation}
for all $(u,w)\in V$.

Now we plug the representations for $\pi_2$, $T\pi_2$ and $(I_{V_1}+iT_{11})^{-1}$ into~\eqref{eq:PVa}.
One easily verifies that $WZQ=i(W-Q)$. Then by a straightforward computation it follows from~\eqref{eq:PVa} that
\[
    \Pi(u,w) = \bigl(u, (I-iWZ)Pw - uWY\bigr)
\]
for all $(u,w)\in V$, where $\Pi$ is as in Proposition~\ref{prop:areg-abstract}.
Let $(u_1,w_1), (u_2,w_2)\in V$. Then
\begin{align}\label{eq:a-reg}
\begin{split}
    \mathrlap{\form{\tilde{a}}\bigl(\Pi(u_1,w_1),\Pi(u_2,w_2)\bigr)}\qquad & \\
        &= \scalar{(I+iZW^*)(I+iZ)(I-iWZ)Pw_1}{Pw_2} \\
        &\quad {} - \scalar{Q(I+iZ)(I-iWZ)Pw_1}{u_2WY} + \scalar{(I-iWZ)Pw_1}{u_2X} \\
        &\quad {} - \scalar{u_1Y}{W^*(I-iZ)(I-iWZ)Pw_2} + \scalar{u_1Y}{(I-iWZ)Pw_2} \\
        &\quad {} + \scalar{u_1Q(I+iZ)WY}{u_2WY} - \scalar{u_1WY}{u_2X} - \scalar{u_1QY}{u_2WY} + \scalar{c_0u_1}{u_2}.
\end{split}        
\end{align}
We simplify~\eqref{eq:a-reg}.
Using the identity $QZW=i(W-Q)$, one deduces that
\begin{equation}\label{eq:fo-w1u2}
    Q(I+iZ)(I-iWZ)P=0.
\end{equation}
This simplifies the first-order terms in~\eqref{eq:a-reg} that involve $w_1$ and $u_2$.
Using~\eqref{eq:fo-w1u2} and $PW=0$, one establishes that
\[
    P(I+iZW^*)(I+iZ)(I-iWZ)P = P(I+iZ + ZWZ)P.
\]
This simplifies the second-order terms in~\eqref{eq:a-reg}.
One readily verifies that $2W^*W = W^*+W$.
Then, also using $QZW=i(W-Q)$ and $W^*P=0$, it follows that  
\[
    -W^*(I-iZ)(I-iWZ)P + (I-iWZ)P = \bigl(I + i(2W^*W-W)Z\bigr)P =  (I + iW^*Z)P.
\]
This simplifies the first-order terms that involve $u_1$ and $w_2$.
Finally, for the terms involving $u_1$ and $u_2$, observe that
\[
    Q(I+iZ)W - Q = 0.
\]
Now the theorem follows from~\eqref{eq:a-reg} and Proposition~\ref{prop:areg-abstract}.
\end{proof}

Theorem~\ref{thm:formula} shows that the regular part $\form{a}_\reg$ is indeed a differential sectorial form. The most remarkable aspect of~\eqref{eq:genformula} is the appearance of the
zeroth-order term involving $X$ and $Y$, even if $c_0=0$, i.e., if the original form $\form{a}$ did not have a zeroth-order term. 
This new zeroth-order term can affect the vertex of the singular part.
A simple concrete example where this happens is given in Example~\ref{ex:vertex-shift}.

\section{About the sectoriality of the singular part}

In this section we first characterise in the setting of Proposition~\ref{prop:areg-abstract} when the singular part $\form{a}_\sing=\form{a}-\form{a}_\reg$ is sectorial.
In~\cite[Proposition~3.1]{tES2011} we proved that $\formreal(\form{a}_\reg)=(\formreal\form{a})_\reg$ if and only if $T\pi_2=0$.
Moreover, if $\form{a}$ is a pure second-order differential sectorial form satisfying Condition~(L) or (B),
then this is also equivalent to $\form{a}_\sing$ being sectorial, cf.~\cite[Corollary~4.3]{tES2011}.
We shall see that for differential sectorial forms the presence of lower-order terms can lead to a more diverse behaviour than possible in the pure second-order
case.

\begin{lemma}\label{lem:as-sect-char}
Assume the notation and conditions of Proposition~\ref{prop:areg-abstract}.
Then $\form{a}_\sing=\form{a}-\form{a}_\reg$ is sectorial
if and only if there exists an $M>0$ such that
\begin{equation}\label{eq:est8}
    \norm{T\pi_2\Phi(u)}_V^2 \le M\norm{u}^2_H
\end{equation}
for all $u\in\formdom{a}$.
\end{lemma}
\begin{proof}
If $\form{a}_\sing$ is sectorial, then~\eqref{eq:est8} follows from~\cite[(8)]{tES2011} and the identity $\tilde{j}(\pi_2\Phi(u))=u$
for all $u\in\formdom{a}$.

For the converse, note that we may assume without loss of generality that $\norm{u}^2_V=\form{\tilde{h}}(u,u)$ for all $u\in V$.
By the formula for $\formreal(\form{a}_\sing)$ in the proof of~\cite[Proposition~3.4]{tES2011} and~\eqref{eq:est8},
there exists an $\omega_\sing>0$ such that
\begin{equation}\label{eq:Re-asing-est}
    \Re\form{a}_\sing(u,u)\ge \norm{\pi_1\Phi(u)}_V^2 - \omega_\sing\norm{u}_H^2
\end{equation}
for all $u\in\formdom{a}$.
By the formula for $\formimag(\form{a}_\sing)$ in the proof of~\cite[Proposition~3.4]{tES2011}, we have
\begin{multline*}
    \abs{\Im\form{a}_\sing(u,u)} \le \norm{T\pi_1\Phi(u)}_V\norm{\pi_1\Phi(u)}_V 
    + 2\norm{T\pi_2\Phi(u)}_V\norm{\pi_1\Phi(u)}_V \\
        + \norm{T_{11}(I_{V_1}+T^2_{11})^{-1}}\norm{T\pi_2 \Phi(u)}^2_V
\end{multline*}
for all $u\in\formdom{a}$. Using~\eqref{eq:est8} and taking $C>0$ sufficiently large, we obtain
\[
    \abs{\Im\form{a}_\sing(u,u)}\le C\bigl(\norm{\pi_1\Phi(u)}_V^2 + \norm{u}_H^2\bigr).
\]
for all $u\in\formdom{a}$.
By~\eqref{eq:Re-asing-est} this shows that $\form{a}_\sing$ is sectorial.
\end{proof}
\begin{remark}
It is readily verified that, after obvious modifications, Lemma~\ref{lem:as-sect-char}
holds in the general $j$-sectorial setting of~\cite[Section~3]{tES2011}.
\end{remark}

From now on we assume that $\form{a}$ is defined as in Section~\ref{sec:def-form}. Moreover, we assume that $\form{a}$ satisfies the conditions of Theorem~\ref{thm:formula},
i.e., we assume that $\form{a}$ satisfies Condition~\textup{(L)} or~\textup{(B)}.
In the following, we shall use the notation of Theorem~\ref{thm:formula}.

Let $\form{a}^\mathdcl{p}$ be the differential sectorial form that belongs to the pure second-order differential expression
\[
    -\sum_{k,l=1}^d\partial_l c_{kl}\partial_k.
\]
We denote the regular part and the singular part of $\form{a}^\mathdcl{p}$ by $\form{a}^\mathdcl{p}_\reg$ and $\form{a}^\mathdcl{p}_\sing$, respectively.
Observe that
\begin{align}
    \begin{split}\label{eq:ar-formula}
    \form{a}_\reg(u,v) &= \form{a}^\mathdcl{p}_\reg(u,v) 
     + \scalar[b]{(I-iWZ)PA^{1/2}\nabla u}{vX} 
     + \scalar[b]{uY}{(I+iW^*Z)PA^{1/2}\nabla v} \\
     &\qquad {} - \scalar[b]{uWY}{vX} + \scalar{c_0u}{v}
     \end{split}
\end{align}
for all $u,v\in\formdom{a}$.
Therefore with~\eqref{eq:simpform} one deduces that
\begin{align}
    \begin{split}\label{eq:as-formula}
    \form{a}_\sing(u,v) &= \form{a}^\mathdcl{p}_\sing(u,v)
    + \scalar[b]{(Q+iWZP)A^{1/2}\nabla u}{vX} + \scalar[b]{uY}{(Q-iW^*ZP)A^{1/2}\nabla v} \\
    &\qquad {}+ \scalar{uWY}{vX}
    \end{split}
\end{align}
for all $u,v\in\formdom{a}$.

The next lemma is an intermediate result in the proof of~\cite[Corollary~4.3]{tES2011}.
\begin{lemma}\label{lem:QZImQA}
If $QZ(I-Q)A^{1/2}=0$ a.e., then $QZ=ZQ$ a.e.
\end{lemma}

The following result generalises~\cite[Corollary~4.3]{tES2011}.
\begin{proposition}\label{prop:as-sect}
The following statements are equivalent.
\begin{romanenum}
\item\label{en:as-sect} $\form{a}_\sing$ is sectorial.
\item\label{en:QZcommute}  $QZ=ZQ$ a.e.
\item\label{en:ar-form} For all $u,v\in\formdom{a}$ one has
\begin{align*}
    \form{a}_\reg(u,v) &= \scalar{(I+iZ)(I-Q)A^{1/2}\nabla{u}}{(I-Q)A^{1/2}\nabla{v}} \\
        &\qquad {}+ \scalar[b]{\conj{X}^{\,\mathdcl{t}}(I-Q)A^{1/2}\nabla u}{v} + \scalar[b]{u}{\conj{Y}^{\,\mathdcl{t}}(I-Q)A^{1/2}\nabla v} \\
        &\qquad {}- \scalar[b]{(\conj{X}^{\,\mathdcl{t}}Q(I+iZ)^{-1}QY)u}{v} + \scalar{c_0u}{v}.
\end{align*}
\item\label{en:Tpi2} For all $u\in\formdom{a}$ one has 
\[
    T\pi_2\Phi(u) = \bigl(0,-\tfrac{1}{2}uQZQ(X+Y)+\tfrac{i}{2}uQ(X-Y)\bigr).
\]
\item\label{en:aps-sect} $\form{a}^\mathdcl{p}_\sing$ is sectorial.
\end{romanenum}
\end{proposition}
\begin{proof}
\begin{parenum}
\item[`\ref{en:as-sect}$\Rightarrow$\ref{en:QZcommute}':]
Let $\tau\in\Cinfc$ and $\xi\in\RR^d$.
For all $\lambda>0$ define $u_\lambda\in\Cinfc$ by $u_\lambda = e^{i\lambda x\cdot\xi}\tau(x)$. 
Since $\form{a}_\sing$ is sectorial, it follows from~\eqref{eq:est8} that there exists an $M>0$
such that
\[
    \form{\tilde{h}}(T\pi_2\Phi(u_\lambda),T\pi_2\Phi(u_\lambda)) \le M\norm{u_\lambda}^2_{\Ltwo}.
\]
Expanding the terms using~\eqref{eq:form-ht} and~\eqref{eq:Tpi2} gives
\[
    \int_\Omega\abs{QZ(I-Q)A^{1/2}(i\lambda\tau\xi + \nabla\tau) + \tfrac{1}{2}\tau QZQ(X+Y)+\tfrac{i}{2}\tau Q(X-Y)}^2 \le M\norm{\tau}^2_{\Ltwo}.
\]
Dividing both sides by $\lambda^2$ and letting $\lambda\to\infty$ shows that $\tau QZ(I-Q)A^{1/2}\xi=0$ a.e.
This implies that $QZ(I-Q)A^{1/2}=0$ a.e. Now it follows from Lemma~\ref{lem:QZImQA} that $QZ=ZQ$ a.e.
\item[`\ref{en:QZcommute}$\Rightarrow$\ref{en:ar-form}':] 
This follows from~\eqref{eq:ar-formula}.
\item[`\ref{en:ar-form}$\Rightarrow$\ref{en:QZcommute}':] 
This follows from~\cite[Corollary~4.3~`(iii)$\Rightarrow$(ii)']{tES2011} applied to $\form{a}^\mathdcl{p}$.
\item[`\ref{en:QZcommute}$\Rightarrow$\ref{en:Tpi2}':]
This is immediate, using~\eqref{eq:Tpi2}.
\item[`\ref{en:Tpi2}$\Rightarrow$\ref{en:as-sect}':]
This is a consequence of Lemma~\ref{lem:as-sect-char}.
\item[`\ref{en:QZcommute}$\Leftrightarrow$\ref{en:aps-sect}':]
This follows from~\cite[Corollary~4.3~`(ii)$\Leftrightarrow$(iv)']{tES2011}.
\qedhere
\end{parenum}
\end{proof}
\begin{remark}
Suppose that $\form{a}_\sing$ is sectorial. Then by Proposition~\ref{prop:as-sect}~`\ref{en:as-sect}$\Rightarrow$\ref{en:QZcommute}' and~\eqref{eq:as-formula} 
it follows that
\[
    \form{a}_\sing(u,v) = \form{a}^\mathdcl{p}_\sing(u,v)
    + \scalar[b]{QA^{1/2}\nabla u}{vX} + \scalar[b]{uY}{QA^{1/2}\nabla v} 
    + \scalar{uWY}{vX}
\]
for all $u,v\in\formdom{a}$. 
\end{remark}

Next we characterise when the regular part of the real part equals the real part of the regular part.
\begin{lemma}
We have $(\formreal\form{a})_\reg=\formreal(\form{a}_\reg)$ if and only if both $QZ=ZQ$ and $(I+iZ)QX=(I-iZ)QY$ a.e.
\end{lemma}
\begin{proof}
By~\cite[Proposition~3.1]{tES2011}, we know that $(\formreal\form{a})_\reg=\formreal(\form{a}_\reg)$ if and only if $T\pi_2 = 0$.
\begin{asparaenum}
\item[`$\Rightarrow$':]
Suppose $T\pi_2=0$. 
By Lemma~\ref{lem:as-sect-char} the form $\form{a}_\sing$ is sectorial.
Therefore it follows from Proposition~\ref{prop:as-sect}~`\ref{en:as-sect}$\Rightarrow$\ref{en:QZcommute}' and~`\ref{en:as-sect}$\Rightarrow$\ref{en:Tpi2}'
that $QZ=ZQ$ and $iQ(X-Y) = QZQ(X+Y)$ a.e.
Now the claim follows by rearranging the terms.
\item[`$\Leftarrow$':]
After rearranging terms, we obtain $iQ(X-Y)=QZQ(X+Y)$ a.e. and $QZ=ZQ$ a.e.
Hence it follows directly from~\eqref{eq:Tpi2} that $T\pi_2=0$.\qedhere
\end{asparaenum}
\end{proof}

We finish with an example that shows that $\form{a}_\sing$ can be sectorial while at the same time $(\formreal\form{a})_\reg\ne \formreal(\form{a}_\reg)$.
Moreover, the example shows that if $\gamma$ is a vertex for $\form{a}$, then $\gamma$ needs not to be a vertex for $\form{a}_\sing$.
Both phenomena do not occur for differential sectorial forms that are purely of second order. 
\begin{example}\label{ex:vertex-shift}
Let $K\subset [0,1]$ be a compact set with empty interior and strictly positive Lebesgue measure $\meas{K}$.
Consider the form $\form{a}\colon\Hone[\RR]\times\Hone[\RR]\to\CC$ given by
\[
    \form{a}(u,v) = \int_\RR\one_K u'\conj{v'} + \int_\RR\one_Ku'\conj{v} - \int_\RR\one_K u\conj{v'} + \int_\RR\one_K u\conj{v}.
\]
Then $\form{a}$ is sectorial in $\Ltwo[\RR]$. More precisely,
\[
    (\formreal\form{a})(u,v)=\int_\RR\one_K u'\conj{v'} + \int_\RR\one_K u\conj{v}
\]
and
\[
    \abs{\Im\form{a}(u,u)}\le \Re\form{a}(u,u) 
\]
for all $u,v\in\Hone[\RR]$. So $\form{a}$ has vertex $0$.
It follows from~\cite[Lemma~4.4]{tES2011} that we may take $Q=\one_K$.
Clearly $Z=0$, so $\form{a}_\sing$ is sectorial by Proposition~\ref{prop:as-sect}.
Using the formula in Proposition~\ref{prop:as-sect}\,\ref{en:ar-form}, we obtain
\[
    \form{a}_\reg(u,v)=2\int_\RR \one_K u\conj{v}
\]
and hence
\[
    \form{a}_\sing(u,v)= \int_\RR\one_K u'\conj{v'} + \int_\RR\one_Ku'\conj{v} - \int_\RR\one_K u\conj{v'} - \int_\RR\one_K u\conj{v}
\]
for all $u,v\in\Hone[\RR]$. It is easily seen that 
\[
    \formreal(\form{a}_\reg) = \form{a}_\reg \ne \tfrac{1}{2}\form{a}_\reg = (\formreal\form{a})_\reg.
\]
Now let $u\in\Cinfc[\RR]$ be such that $\restrict{u}{[0,1]}=1$. Then $\Re\form{a}_\sing(u,u)=-\meas{K}< 0$.
This shows that $0$ is not a vertex of $\form{a}_\sing$.

Finally, if $\form{b}\colon\Hone[\RR]\times\Hone[\RR]\to\CC$ is the form
without zeroth-order term given by
\[
    \form{b}(u,v) = \int_\RR\one_K u'\conj{v'} + \int_\RR\one_Ku'\conj{v} - \int_\RR\one_K u\conj{v'},
\]
then
\[
    \form{b}_\reg(u,v)=\int_\RR \one_K u\conj{v}
\]
for all $u,v\in\Hone[\RR]$ and $\form{b}_\reg$ contains a non-trivial zeroth-order term.
\end{example}

\subsection*{Acknowledgements} 
We would like to thank El-Maati Ouhabaz for raising the question whether 
the formula for the regular part can be extended to allow lower-order terms.
The authors thank the referee for his comments which improved the presentation of the paper.
Part of this work is supported by the Marsden Fund Council from Government
funding, administered by the Royal Society of New Zealand.

%\bibliographystyle{amsalphaabbrv}
%\bibliography{references}

\end{document}